\documentclass[11pt]{article}
\usepackage{amsfonts, amsmath, amssymb, amscd, amsthm, color, graphicx, mathrsfs, mathabx, wasysym, setspace, mdwlist, calc,float}
\usepackage{setspace}
\usepackage{hyperref}
\usepackage{tikz-cd}
\usepackage{hyperref}
\usepackage{tocloft}
\usepackage{xcolor}

\usepackage{hyperref}
\hypersetup{linktocpage}

\hypersetup{colorlinks,
    linkcolor={red!50!black},
    citecolor={blue!80!black},
    urlcolor={blue!80!black}}
\usepackage{float}

\renewcommand{\phi}{\varphi}
\newcommand{\e}{\varepsilon}
\newcommand{\NN}{\mathbb{N}}

\newcommand{\ZZ}{\mathbb{Z}}

\renewcommand{\d}{{\rm d}}
\renewcommand{\ll }{\langle\hspace{-.7mm}\langle }
\newcommand{\rr }{\rangle\hspace{-.7mm}\rangle }

\newcommand{\Aut}{\operatorname{Aut}}
\newcommand{\Inn}{\operatorname{Inn}}
\newcommand{\Out}{\operatorname{Out}}

\newcommand{\Ker}{\operatorname{Ker}}

\newtheorem{thm}{Theorem}[section]
\newtheorem{cor}[thm]{Corollary}
\newtheorem{lem}[thm]{Lemma}
\newtheorem{prop}[thm]{Proposition}
\newtheorem{q}[thm]{Question}

\theoremstyle{definition}
\newtheorem{defn}[thm]{Definition}

\theoremstyle{remark}
\newtheorem{rem}[thm]{Remark}

\renewcommand{\L}{\mathcal{L}}

\newcommand{\Mod}{{\rm Mod}}

\newcommand{\WR}{\mathcal{WR}}
\newcommand{\G}{\mathcal{G}}

\begin{document}

\title{$Out(F_n)$-invariant probability measures on the space of $n$-generated marked groups}
\date{}
\author{D. Osin}

\maketitle

\vspace{-3mm}

\begin{abstract}%\blfootnote{\textbf{MSC} Primary: 20F65. Secondary: 20F67}
Let $\G_n$ denote the space of $n$-generated marked groups. We prove that, for every $n\ge 2$, there exist $2^{\aleph_0}$ non-atomic, $\Out(F_n)$-invariant, mixing probability measures on $\G_n$. On the other hand, there are non-empty closed subsets of $\G_n$ that admit no $\Out(F_n)$-invariant probability measure. Acylindrical hyperbolicity of the group $\Aut(F_n)$ plays a crucial role in the proof of both results. We also discuss model theoretic implications of the existence of $\Out(F_n)$-invariant, ergodic probability measures on $\G_n$.
\end{abstract}

\section{Introduction}

Let $n$ be a natural number. An \emph{$n$-generated marked groups} is a pair $(G,X)$, where $G$ is a group and $X\in G^n$  is an ordered generating set of $G$. Two marked groups $(G, (x_1, \ldots, x_n))$ and $(H, (y_1, \ldots, y_n))$ are \emph{isomorphic} if there exists a group isomorphism $G\to H$ that sends $x_i$ to $y_i$ for all $i=1, \ldots, n$.

The set of isomorphism classes of $n$-generated marked groups, denoted by $\G_n$, can be naturally identified with normal subgroup of $F_n$, the free group of rank $n$. Namely, a marked group $(G,X)$ corresponds to the kernel of the natural homomorphism $F_n\to G$ induced by mapping a fixed basis of $F_n$ to $X$. The product topology on $2^{F_n}$ induces the structure of a compact Polish space on the set of normal subgroups of $F_n$ and, via the  above identification, on $\mathcal G_n$.

In the context of combinatorial group theory, the usefulness of this topological approach was demonstrated by Grigorchuk in \cite{Gri84}. Ever since, the space $\G_n$ has played an important role in the study of algebraic, geometric, and model theoretic properties of groups. For more details, we refer the reader to \cite{Cha,CG,Gri05,Osi21} and references therein.

The action of $\Aut(F_n)$ on $F_n$ gives rise to an action of $\Out(F_n)$ on $\mathcal G_n$ by homeomorphisms, which preserves the (unmarked) isomorphism relation; that is, if $(G,X)$ and $(H,Y)$ belong to the same $Out(F_n)$-orbit in $\G_n$, then $G\cong H$. In this paper, we study the existence of invariant Borel probability measures in this dynamical system. Specifically, we address the following basic question asked by Grigorchuk during a talk at the conference ``Self-similarity of groups, trees and fractals" in June 2022 (a similar problem is discussed in the paper \cite{Gri05}).

\begin{q}\label{p1}
Does there exist an $\Out(F_n)$-invariant, non-atomic, ergodic probability measure on $\G_n$?
\end{q}

\begin{rem}\label{rem} Atomic $\Out(F_n)$-invariant probability measures on $\G_n$ are easy to construct. Indeed, the action of $\Out(F_n)$ on $\G_n$ fixes all points corresponding to characteristic subgroups of $F_n$. Clearly, the Dirac measures supported on these points are $\Out(F_n)$-invariant. More generally, we can define probability measures supported on finite orbits (e.g., orbits of finite marked groups).

Furthermore, Olshanskii's solution of the finite basis problem for group varieties \cite{Ols} implies the existence $2^{\aleph_0}$ points in $\G_n$ fixed by $\Out(F_n)$. Since the set of fixed points is closed, it contains a non-empty closed subset $\mathcal C$ without isolated points by the Cantor-Bendixson theorem. Recall that the set of non-atomic probability measures supported on a compact metrizable space $X$ without isolated points is dense in the space of all probability measures on $X$ equipped with the weak$^\ast$ topology. Therefore, there exist plenty of non-atomic, $\Out(F_n)$-invariant measures on $\G_n$ supported on $\mathcal C$. However, all these measures are not ergodic.
\end{rem}

Apart from being natural, Question \ref{p1} is also motivated by the following observation: \emph{if $\mu$ is an $\Out(F_n)$-invariant, ergodic probability measure on $\G_n$, then $\mu$-generic marked groups are ``virtually indistinguishable"}. The precise formulation of this claim involves infinitary model theory and is given in Section \ref{Sec:MT}. In particular, every non-atomic, $\Out(F_n)$-invariant, ergodic probability measure on $\G_n$ gives rise to an uncountable set of finitely generated, pairwise non-isomorphic, elementarily equivalent groups (see Proposition \ref{Prop:EE} and Remark \ref{Rem:c}). Examples of this kind are of interest since the standard tools for constructing elementarily equivalent models, such as ultrapowers and the L\"owenheim-Skolem theorem, are incapable of producing finitely generated structures.

As shown below, the answer to Question \ref{p1} is affirmative. This prompts yet another natural problem.

\begin{q}\label{p2}
Does every subsystem of $\Out(F_n)\curvearrowright \mathcal G_n$ admit an $Out(F_n)$-invariant probability measure?
\end{q}

By a \emph{subsystem} of the dynamical system $\Out(F_n)\curvearrowright \mathcal G_n$ we mean a non-empty, closed, $Out(F_n)$-invariant subset of $\G_n$. Note that the ``obvious" probability measures discussed in Remark \ref{rem} can be avoided by passing to minimal subsystems containing infinitely many points.

The main goal of this paper is to answer both questions. For a group $G$, we denote by $2^G$ the power set of $G$ endowed with the product topology. The action of $G$ on itself by left multiplication induces a continuous $G$-action on $2^G$. Our approach is based on the following theorem of independent interest.

\begin{thm} \label{main1}
For every integer $n\ge 2$, there exists a continuous, injective, $Out(F_n)$-equivariant map $2^{\Out(F_n)}\to \G_n$.
\end{thm}

Theorem \ref{main1} allows us to translate various dynamical phenomena occurring in $2^{\Out(F_n)}$ to $\G_n$. Its proof makes use of acylindrical hyperbolicity of the group $\Aut(F_n)$ established in \cite{GH} and certain results obtained via group theoretic Dehn filling in \cite{CIOS}.

It is well-known that the dynamical system $G\curvearrowright 2^G$ has $2^{\aleph_0}$  non-atomic, $G$-invariant, mixing measures. Utilizing Theorem \ref{main1}, we obtain the affirmative answer to Question \ref{p1}.

\begin{cor} \label{Cor1}
For any $n\ge 2$, there exist $2^{\aleph_0}$  pairwise distinct, non-atomic, $\Out(F_n)$-invariant, mixing (in particular, ergodic), probability measures on $\mathcal G_n$.
\end{cor}

On the other hand, for every non-amenable group $G$, there exists a closed $G$-invariant subset of $2^G$ without any $G$-invariant probability measure (see \cite[Theorem 1.2.2]{FKSV}). Since $Out(F_n)$ is non-amenable for all $n\ge 2$, we get the negative answer to Question \ref{p2}.

\begin{cor} \label{Cor2}
For any $n\ge 2$, the dynamical system $\Out(F_n)\curvearrowright\G_n$ contains a subsystem that does not admit any $\Out(F_n)$-invariant probability measure.
\end{cor}

The paper is organised as follows. In the next section, we collect some basic definitions and necessary results on acylindrically hyperbolic groups. In Section \ref{Sec:AH}, we prove a generalization of Theorem \ref{main1} and derive Corollaries \ref{Cor1} and \ref{Cor2}. Section~\ref{Sec:MT} is devoted to the discussion of model theoretic aspects of ergodic measures on $\G_n$.

\paragraph{Acknowledgments.} The author is grateful to Rostislav Grigorchuk and Vadim Kaimanovich for bringing Questions \ref{p1} and \ref{p2} to his attention and stimulating discussions. This work has been supported by the NSF grant DMS-1853989.

\section{Preliminaries}

\paragraph{2.1. Dynamical systems.}
Let $G\curvearrowright X$ be a topological dynamical system; that is, $X$ is a topological space  endowed with a continuous action of a group $G$. By a \emph{subsystem} of $G\curvearrowright X$ we mean any closed subset of $X$ endowed with the induces action of $G$.

A \emph{probability measure} on $X$ is a countably additive measure $\mu$ defined on the $\sigma$-algebra of all Borel subsets of $X$ such that $\mu(X)=1$. A probability measure $\mu$ on $X$ is \emph{$G$-invariant} if $\mu(gA)=\mu(A)$ for every Borel subset $A\subseteq X$ and any $g\in G$. Further, a $G$-invariant measure $\mu$ on $X$ is said to be
\vspace{-3mm}
\begin{itemize}\setlength\itemsep{1.2mm}
\item \emph{non-atomic} if $\mu(x)=0$ for all $x\in X$;

\item \emph{ergodic} if $\mu(A)\in \{ 0, 1\}$ for every $G$-invariant Borel set $A\subseteq X$;

\item \emph{mixing} if, for any Borel $A,B\subseteq X$ and any $\e>0$, there exists a finite subset $S\subseteq G$ such that $|\mu(A \cap gB)-\mu(A)\mu(B)|<\e$ for all $g\in G\setminus S$.
\end{itemize}
\vspace{-3mm}
\noindent Obviously, every mixing measure is ergodic.

For an abstract set $I$, let $2^I$ denote the set of all subsets of $I$ endowed with the product topology. Recall that the {\it product topology} on $2^I$ is simply the topology of pointwise convergence of indicator functions. Equivalently, it can be defined by taking the family of sets $U(S,F)=\{ T\in 2^I\mid T\cap F=S\cap F\}$, where $F$ ranges in the set of all finite subsets of $I$, as the base of neighborhoods of the point $S\in 2^I$. By Tychonoff's theorem, $2^I$ is compact. For any measure $\mu$ on $\{ 0,1\}$, we denote by $\mu^I$ the product measure on $2^I$ (formally speaking, we identify the power set of $I$ with $\{ 0,1\}^I$ here).

Dynamical systems of particular interest to us are \emph{generalized Bernoulli shifts} $G\curvearrowright 2^{G/H}$, where $H$ is a subgroup of $G$ and $G/H$ denotes the set of left cosets. The action of the group $G$ on $G/H$ by left multiplication induces a continuous action of $G$ on $2^{G/H}$. If $H=\{ 1\}$ we obtain the ordinary Bernoulli shift $G\curvearrowright 2^G$. We will need the following.

\begin{lem}[{\cite[Proposition 2.3]{KT} and \cite[Theorem 1.2.2]{FKSV}}]\label{Lem:K}
Let $G$ be a countably infinite group.
\begin{enumerate}
\item[(a)]  If $\mu$ is a measure on $\{ 0, 1\}$ that does not concentrate on a single point, then the product measure $\mu^G$ on $2^G$ is mixing.
\item[(b)] The group $G$ is amenable if and only if every subsystem of $G\curvearrowright 2^G$ admits an invariant probability measure.
\end{enumerate}
\end{lem}

\paragraph{2.2. Acylindrically hyperbolic groups.}
An isometric action of a group $G$ on a metric space $S$ is {\it acylindrical} if for every $\e>0$ there exist $R,N>0$ such that for every two points $x,y\in S$ with $\d (x,y)\ge R$, there are at most $N$ elements $g\in G$ satisfying
$$
\d(x,gx)\le \e \;\;\; {\rm and}\;\;\; \d(y,gy) \le \e.
$$

An action of a group $G$ on a hyperbolic space $S$ is \emph{non-elementary} if the limit set of $G$ on the Gromov boundary $\partial S$ has infinitely many points; for acylindrical actions, this condition is equivalent to the requirement that $G$ is not virtually cyclic and the action has infinite orbits {\cite[Theorem 1.1]{Osi16}}. Every group has an acylindrical action on a hyperbolic space, namely the trivial action on the point. For this reason, we want to avoid elementary actions in the definition below.

\begin{defn}
A group is \emph{acylindrically hyperbolic} if it admits a non-elementary acylindrical action (by isometries) on a hyperbolic space.
\end{defn}

The class of acylindrically hyperbolic groups is rather wide and includes all non-elementary hyperbolic and relatively hyperbolic groups, $Out(F_n)$ for $n\ge 2$, mapping class groups of closed surfaces of genus at least $2$, finitely presented groups of deficiency at least $2$, most $3$-manifold groups, and many other examples. For details, we refer the reader to \cite{Osi18,Osi16} and references therein. An example of particular importance for us is the group $Aut(F_n)$, whose acylindrical hyperbolicity was recently proved by Genevois and Horbez \cite{GH}.

Below, we summarize some results about acylindrically hyperbolic groups from \cite{CIOS,DGO,Hull} necessary for the proof of Theorem \ref{main1}. Note that acylindrically hyperbolic groups appear in \cite{DGO} under the name of ``groups with non-degenerate hyperbolically embedded subgroups". The equivalence of this condition to acylindrical hyperbolicity was established later in \cite{Osi16}.

\begin{lem}[{\cite[Theorem 2.24]{DGO}}]\label{Lem:KG}
Every acylindrically hyperbolic group $G$ contains a unique maximal finite normal subgroup.
\end{lem}

The maximal finite normal subgroup of an acylindrically hyperbolic group $G$ is called the \emph{finite radical} of $G$ and is denoted by $K(G)$.

A subgroup $H$ of an acylindrically hyperbolic group $G$ is said to be \emph{suitable} if there is a generating set $X$ of $G$ such that the Cayley graph $Cay(G,X)$ is hyperbolic, the action of $H$ on $Cay(G,X)$ is non-elementary, and $H$ does not normalize any non-trivial, finite, normal subgroup of $G$ (see \cite[Definition 1.4]{Hull}). The next lemma is well-known to experts, although it does not seem to have been recorded in the literature.

\begin{lem}\label{Lem:suit}
If $G$ is an acylindrically hyperbolic group with $K(G)=\{ 1\}$, then every non-trivial normal subgroup of $G$ is suitable.
\end{lem}

\begin{proof}
Let $M$ be a non-trivial normal subgroup of $G$. Note that $M$ must be infinite as $K(G)=\{ 1\}$. By \cite[Theorem 1.2]{Osi16}, there exists a generating set $X$ of $G$ such that the Cayley graph $Cay(G,X)$ is hyperbolic and the action of $G$ on $Cay(G,X)$ is non-elementary. Since $M$ is normal in $G$ and infinite, the induced action of $M$ on $Cay(G,X)$ is non-elementary by \cite[Lemma 7.1]{Osi16}. Further, by \cite[Lemma 5.5]{Hull}, this implies the existence of a finite subgroup $L\le G$ normalized by $M$ such that every other finite subgroup of $G$ normalized by $M$ is contained in $L$. Since $M\lhd G$, the conjugates $g^{-1}Lg$ are also normalized by $M$ for all $g\in G$.  Therefore, $g^{-1}Lg\le L$ for all $g\in G$. This implies that $L\lhd G$ and hence $L\le K(G)=\{ 1\}$. Thus, $M$ is suitable.
\end{proof}

An element $g$ of an acylindrically hyperbolic group $G$ is called \emph{loxodromic} if there exists an acylindrical action of $G$ on a hyperbolic space $S$ such that $\langle g\rangle $ has unbounded orbits (for details and equivalent definitions, see \cite{Osi16}).

\begin{lem}[{\cite[Corollary 2.9]{DGO}}]\label{Lem:Eg}
Let $G$ be an acylindrically hyperbolic group. Every loxodromic element $g\in G$ is contained in a unique maximal virtually cyclic subgroup of $G$.
\end{lem}

The maximal virtually cyclic subgroup of an acylindrically hyperbolic group $G$ containing a loxodromic element $g\in G$ is denoted by $E(g)$. By \cite[Corollary 5.7]{Hull}, every suitable subgroup of $G$ contains  a loxodromic element $g\in G$ such that $E(g)=\langle g\rangle$. Combining this with Lemma \ref{Lem:suit} yields the following.

\begin{lem}\label{Lem:lox}
Let $G$ be an acylindrically hyperbolic group with $K(G)=\{ 1\}$. Every non-trivial normal subgroup of $G$ contains a loxodromic element $g\in G$ such that $E(g)=\langle g\rangle$.
\end{lem}

\paragraph{2.3. Wreath-like products of groups.}
Our proof of Theorem \ref{main1} makes use of the notion of a wreath-like product of groups introduced in \cite{CIOS}. We recall the definition here.

\begin{defn}\label{Def:WR}
Let $A$, $Q$ be arbitrary groups, $I$ an abstract set, $Q\curvearrowright I$ a (left) action of $Q$ on $I$. A group $W$ is called a \emph{wreath-like product} of groups $A$ and $Q$ corresponding to the action $Q\curvearrowright I$ if $W$
is an extension of the form
\begin{equation}\label{ext}
1\longrightarrow B(W) \longrightarrow  W \stackrel{\e}\longrightarrow Q\longrightarrow 1,
\end{equation}
where
\begin{equation}\label{Eq:BWdef}
B(W)=\bigoplus_{i\in I}A_i,
\end{equation}
$A_i\cong A$ for all $i\in I$, and the action of $W$ on $B(W)$ by conjugation satisfies the rule
\begin{equation}\label{Eq:wAi}
wA_iw^{-1} = A_{\e(w)i}
\end{equation}
for all $i\in I$. The subgroup $B(W)$ is called the \emph{base} of the wreath-like product $W$ and the map $\e\colon W\to Q$ is called the \emph{canonical homomorphism} associated to the wreath-like structure of $W$. The set of all wreath-like  products of groups $A$ and $Q$ corresponding to the action $Q\curvearrowright I$ is denoted by $\WR(A, Q\curvearrowright I)$.
\end{defn}

The following lemma can be found in \cite{CIOS}; in fact, it easily follows from the description of kernels of group theoretic Dehn fillings obtained in \cite{Sun}. For an element $g$ of a group $G$, we denote by $\ll g\rr$ the smallest normal subgroup of $G$ containing $g$.

\begin{lem}[{\cite[Theorem 2.6]{CIOS}}]\label{Lem:WR}
Let $G$ be an acylindrically hyperbolic group, $g\in G$ a loxodromic element. Suppose that $d$ is a natural number such that $\langle g^d\rangle \lhd E(g)$. Then, for any sufficiently large $k\in \mathbb N$ divisible by $d$, we have $$G/[\ll g^{k}\rr, \ll g^{k}\rr]\in \WR (\ZZ, G/\ll g^k\rr \curvearrowright I),$$ where $I$ is the set of cosets $G/E(g)\ll g^k\rr$ and the action $G/\ll g^k\rr\curvearrowright I$ is by left multiplication.
\end{lem}

\section{Proofs of the main results} \label{Sec:AH}

\paragraph{3.1. Dynamical systems associated to normal subgroups of acylindrically hyperbolic groups.} For a group $G$, we denote by $\mathcal N(G)$ the subspace of $2^G$ consisting of all normal subgroups of $G$. It is easy to see that $\mathcal N(G)$ is closed in $2^G$ and, therefore, is a Polish space with respect to the induced topology. The base of neighborhoods of a subgroup $N\in \mathcal N(G)$ is given by the sets
\begin{equation}\label{Eq:Base}
U_G(N, F)= \{ M\lhd G \mid M\cap F = N\cap F\},
\end{equation}
where $F$ ranges in the set of all finite subsets of $G$.

We begin with auxiliary results.

\begin{lem}\label{Lem:iota}
Let $G$ be a group, $J$ an abstract set. Suppose that $G$ contains a subgroup of the form $\bigoplus_{j\in J} N_j$, where all $N_j$ are non-trivial normal subgroups of $G$. Then the map $\iota\colon 2^J\to \mathcal N(G)$ defined by $\iota(S)=\bigoplus_{j\in S} N_j$ for all $S\subseteq J$ is injective and continuous.
\end{lem}

\begin{proof}
We first note that the image of every $S\subseteq J$ is indeed a normal subgroup of $G$ since so is every $N_j$. Injectivity of the map $\iota$ follows immediately from the assumption that all subgroups $N_j$ are nontrivial.

Further, fix an arbitrary $S\subseteq J$ and an arbitrary finite subset $F$ of $G$. To prove the continuity of $\iota$, it suffices to show that there exists a neighborhood $V$ of $S$ in $2^J$ such that $\iota (V) \subseteq U_G(\iota(S), F)$, where $U_G(\iota (S),F)$ is defined by (\ref{Eq:Base}). Since $F$ is finite, there exists a finite subset $K\subseteq J$ such that
\begin{equation}\label{Eq:F}
F\subseteq \left\langle \bigcup_{j\in K} N_j\right\rangle.
\end{equation}
Let $V=\{ T\subseteq I\mid T\cap K=S\cap K\}$. For every $T\in V$, we obviously have
$$
\iota(T)\cap \left\langle \bigcup_{j\in K} N_j\right\rangle = \bigoplus_{j\in T\cap K} N_j = \bigoplus_{j\in S\cap K} N_j = \iota(S)\cap \left\langle \bigcup_{j\in K} N_j\right\rangle.
$$
Combining this with (\ref{Eq:F}), we obtain $\iota (T)\cap F=\iota (S)\cap F$, i.e., $\iota (T)\in U_G(\iota(S), F)$.
\end{proof}

\begin{lem}\label{Lem:ind}
Let $G$, $H$ be any groups, $\gamma\colon G\to H$ a surjective homomorphism. For a subset $S\subseteq H$, we denote by $\gamma^{-1}(S)$ its full preimage in $G$. The map $\widehat\gamma\colon \mathcal N(H)\to \mathcal N(G)$ defined by the rule $\widehat\gamma(N)=\gamma^{-1}(N)$ for all $N\lhd H$ is continuous.
\end{lem}

\begin{proof}
Let $F$ be a finite subset of $G$ and let $N\lhd H$. For $M\lhd H$, we have $M\cap \gamma(F)=N\cap \gamma(N)$ if and only if $\gamma^{-1}(M)\cap F=\gamma^{-1}(N)\cap F$. Therefore, $\widehat \gamma $ sends $U_H(N, \gamma(F))$ to $U_G(\widehat\gamma(N), F)$ and the result follows.
\end{proof}

\begin{lem}\label{Lem:WRIRS}
Let $W\in \WR(A, Q\curvearrowright I)$, where $A\ne \{ 1\}$ and the action $Q\curvearrowright I$ is transitive. Suppose that $R$ is a normal subgroup of $W$ containing the normalizer $N_W(A_i)$ for some $i\in I$. Then there exists a continuous, injective, $W$-equivariant map $2^{W/R}\to \mathcal N(R)$.
\end{lem}

\begin{proof}
Throughout the proof, we use the notation and terminology introduced in Definition~\ref{Def:WR}. Fix $i\in I$ such that $N_W(A_i)\le R$. Obviously, we have $B(W)\le N_W(A_i)\le R$. For every $wR\in W/R$, we define
$$
N_{wR}=\left\langle \bigcup_{r\in wR} A_{\e(r)i}\right\rangle = \bigoplus_{r\in wR} A_{\e(r)i} \le B(W),
$$
where $\e\colon W\to Q$ is the canonical homomorphism. Using (\ref{Eq:wAi}), we obtain
\begin{equation}\label{tDt}
tN_{wR}t^{-1}= \left\langle \bigcup_{r\in wR} tA_{\e(r)i}t^{-1}\right\rangle = \left\langle \bigcup_{r\in wR} A_{\e(tr)i}\right\rangle
=\left\langle \bigcup_{s\in twR} A_{\e(s)i}\right\rangle =N_{twR}.
\end{equation}
If $t\in R$, then $twR=w(w^{-1}tw)R=wR$ for every $w\in W$. Hence, $N_{wR}\lhd R$ for all $wR\in W/R$.

Suppose that $A_{\e(u)i}=A_{\e(v)i}$ for some $u,v\in W$. Using (\ref{Eq:wAi}) again, we obtain $u^{-1}v\in N_W(A_i)\le R$ and, therefore, $uR=vR$. This implies that $N_{uR}$ and $N_{vR}$ are generated by disjoint sets of direct summands of $B(W)$ whenever $uR\ne vR$. Thus, we have
\begin{equation}\label{Eq:BW}
B(W)=\bigoplus_{wR\in W/R} N_{wR}.
\end{equation}

Let $J=W/R$. Consider the injective, continuous map $\iota \colon 2^{W/R}\to \mathcal N(R)$ defined as in Lemma \ref{Lem:iota}. For any $t\in W$ and any $S\subseteq J=W/R$, we have
$$
\iota(tS) = \left\langle   \bigcup_{wR\in tS} N_{wR}\right\rangle = \left\langle \bigcup_{uR\in S} N_{tu R}\right\rangle = \left\langle \bigcup_{uR\in S} tN_{uR}t^{-1}\right\rangle= t\iota(S)t^{-1}
$$
by (\ref{tDt}). Thus, $\iota$ is $W$-equivariant.
\end{proof}

If $M$ is a normal subgroup of a group $G$, then the action of $G$ on $M$ by conjugation induces a continuous action of $G/M$ on $\mathcal N(M)$. We are now ready to prove the main result of this section. Recall that $K(G)$ denotes the maximal finite normal subgroup of an acylindrically hyperbolic group $G$.

\begin{thm}\label{Thm:main}
Let $G$ be an acylindrically hyperbolic group. For every infinite normal subgroup $M$ of $G$ containing $K(G)$, there exists a continuous, injective, $G$-equivariant map $2^{G/M}\to \mathcal N(M)$.
\end{thm}

\begin{proof}
The group $\overline{G}=G/K(G)$ is acylindrically hyperbolic and has trivial finite radical by \cite[Lemma 3.9]{MO}. Since $M$ is infinite, $\overline{M}=M/K(G)$ is a non-trivial normal subgroup of $\overline{G}$. By Lemma \ref{Lem:lox}, there exists a loxodromic element $g\in \overline{M}$ such that $E(g)=\langle g\rangle$. Further, by Lemma \ref{Lem:WR}, there exists $k\in \mathbb N$ such that the group $W=\overline{G}/[\ll g^{k}\rr, \ll g^{k}\rr]$ belongs to $\WR (\ZZ, \overline{G}/\ll g^k\rr \curvearrowright I)$, where $I=\overline{G}/E(g)\ll g^k\rr$ and the action $\overline{G}/\ll g^k\rr\curvearrowright I$ is by left multiplication.

For brevity, we denote the subgroup $[\ll g^{k}\rr, \ll g^{k}\rr] $ by $D$. Let $R=\overline{M}/D\lhd W$ and let $i=E(g)\ll g^k\rr$ (we think of $i$ as a coset of $E(g)\ll g^k\rr $ in $\overline G$ and thus $i\in I$). In what follows, we employ the notation introduced in Definition \ref{Def:WR}. In particular, we denote by $\e\colon W\to G/\ll g^k\rr$ the canonical homomorphism.

For every $w\in N_W(A_i)$, the element $\e(w)\in G/\ll g^k\rr$ stabilizes $i$ by (\ref{Eq:wAi}). This means that $\e(w)\in E(g)\ll g^k\rr/\ll g^k\rr$. Since $\overline M\lhd \overline G$ and $g\in \overline M$, we obtain
$$
w\in E(g)\ll g^k\rr/D =\langle g\rangle \ll g^k\rr /D\le \overline{M}/D =R.
$$
Thus, $N_W(A_i)\le R$. Applying Lemma \ref{Lem:WRIRS} to the group $W$ and the subgroup $R$, we obtain a continuous, injective, $W$-equivariant map $\iota\colon 2^{W/R}\to \mathcal N(R)$.

Let $\gamma\colon G\to W$ be the composition of the natural homomorphisms $$G\to \overline G=G/K(G)\to W=\overline{G}/D.$$ We endow $W/R$ and $\mathcal N(R)$ with the action of $G$ induced by $\gamma$ and the action of $W$ on the corresponding set. Since $K(G)\le M$ and $g\in \overline M$, we have $\Ker \gamma \le M$. Therefore, $G/M\cong \gamma(G)/\gamma(M)\cong W/R$. This isomorphism induces a $G$-equivariant homeomorphism $\kappa \colon 2^{G/M}\to 2^{W/R}$. Applying Lemma \ref{Lem:ind} to the restriction of $\gamma$ to $M$, we obtain a continuous, injective map  $\widehat\gamma\colon \mathcal N(R) \to \mathcal N(M)$, which is also $G$-equivariant. Thus, the composition $\widehat\gamma\circ\iota \circ \kappa$ has all the required properties.
\end{proof}

\begin{rem}\label{Rem:G/M}
Since the action of $M$ on $2^{G/M}$ and $\mathcal N(M)$ is trivial, we can think of both $2^{G/M}$ and $\mathcal N(M)$ and $G/M$-sets. In these settings, the map constructed in Theorem \ref{Thm:main} is $G/M$-equivariant.
\end{rem}

\paragraph{$\Out(F_n)$-invariant probability measures on $\G_n$.}
Recall that the group $\Aut(F_n)$ is acylindrically hyperbolic for any $n\ge 2$ \cite{GH}. It is also well-known and easy to prove that $\Aut(F_n)$ has no non-trivial finite normal subgroups. Indeed, let $K$ be a finite normal subgroup of $\Aut(F_n)$. We identify $F_n$ with $\Inn(F_n)$ and think of it as a subgroup of $\Aut(F_n)$. Consider $P=\langle F_n, K\rangle $. Since both $F_n$ and $K$ are normal in $P$ and $F_n\cap K=\{ 1\}$, we have $P=F_n\times K$. Hence, the action of $K\le \Aut(F_n)$ on $F_n$ is trivial and we obtain $K=\{ 1\}$.

\begin{proof}[Proof of Theorem \ref{main1}]
Apply Theorem \ref{Thm:main} (and Remark \ref{Rem:G/M}) to the group $G=\Aut(F_n)$ and its normal subgroup $M=F_n$.
\end{proof}

Corollaries \ref{Cor1} and \ref{Cor2} are derived from Theorem \ref{main1} by considering the push-forward and pull-back measures, respectively.

\begin{proof}[Proof of Corollary \ref{Cor1}]
Let $\mu $ be a non-atomic, $\Out(F_n)$-invariant, mixing probability measure on $2^{\Out(F_n)}$. Let $f\colon 2^{\Out(F_n)}\to \G_n$ be a continuous, injective, $\Out(F_n)$-equivariant map provided by Theorem \ref{main1}. For every Borel subset $A\subseteq \G_n$ we define $\nu(A)=\mu(f^{-1}(A))$. It is easy to see that $\nu$ is a non-atomic, $\Out(F_n)$-invariant, mixing probability measure on $\G_n$. It remains to note that there exist $2^{\aleph_0}$ pairwise distinct non-atomic, $\Out(F_n)$-invariant, mixing probability measures on $2^{\Out(F_n)}$ by Lemma \ref{Lem:K} (a) and distinct measures on $2^{\Out(F_n)}$ induce distinct measures on $\G_n$.
\end{proof}

\begin{proof}[Proof of Corollary \ref{Cor2}]
The natural homomorphism $F_n\to F_n/[F_n, F_n]\cong \ZZ^n$ induces an epimorphism $\Out(F_n)\to GL(n,\ZZ)$. It is well-known that $GL(n,Z)$ contains a non-cyclic free subgroup whenever $n\ge 2$. Thus, $GL(n,\ZZ)$ is non-amenable and, therefore, so is $\Out(F_n)$ for $n\ge 2$. By Lemma \ref{Lem:K} (b), there exists a non-empty, closed, $\Out(F_n)$-invariant subset $C\subseteq 2^{\Out(F_n)}$ that admits no $\Out(F_n)$-invariant probability measure. Note that the set $C$ is compact being a closed subset of the compact space $2^{\Out(F_n)}$.

Let $f\colon 2^{\Out(F_n)}\to \G_n$ be a continuous, injective, $G$-equivariant map provided by Theorem \ref{main1} and let $D=f(C)$. Since $f$ is continuous and $\Out(F_n)$-equivariant, $D$ is also compact and $\Out(F_n)$-invariant. Thus, $D$ is a subsystem of $\Out(F_n)\curvearrowright \G_n$. Suppose that there exists an $\Out(F_n)$-invariant probability measure $\nu$ on $D$. Then the rule $\mu(B)=\nu (f(B))$ for all Borel subsets $B\subseteq 2^{\Out(F_n)}$ induces an $\Out(F_n)$-invariant probability measure $\mu$ on $2^{\Out(F_n)}$, which leads to a contradiction.
\end{proof}

\section{Model-theoretic aspects of $\Out(F_n)$-invariant ergodic measures on $\G_n$}\label{Sec:MT}

\paragraph{4.1. Elementary equivalence of generic groups.} Let $\mathcal L$ denote the language of groups. That is, $\mathcal L$ is the first order language with the signature $\{ \cdot, ^{-1}, 1\}$, where $\cdot$ (respectively, $^{-1}$) is a binary (respectively, unary) operation and $1$ is a constant interpreted in the obvious way. Recall that $\L_{\omega_1, \omega}$ is the infinitary version of $\L$,  where countable conjunctions and disjunctions (but only finite sequences of quantifiers) are allowed. For details, we refer the reader to \cite{Mar16}.

The expressive power of $\L_{\omega_1, \omega}$ is much greater than that of $\mathcal L$. In fact, most algebraic, geometric, and even analytic properties of groups can be expressed by $\L_{\omega_1, \omega}$-sentences. Examples include finiteness, solvability, hyperbolicity, amenability, property (T) of Kazhdan, exponential and subexponential growth, etc. (see \cite{Osi21} and references therein). Note that none of these properties can be defined by first order sentences.

Given any countable set $F$ of sentences in $\L_{\omega_1, \omega}$, we say that a group $G$ is \emph{$F$-equivalent} to a group $H$ and write $G\equiv_F H$ if $G$ and $H$ satisfy the same sentences from $F$; that is, for every $\sigma\in F$, we have $G\models \sigma $ if and only if $H\models \sigma$. If $F$ is the set of all first order sentences, $\equiv _F$ becomes the familiar \emph{elementary equivalence} relation.

The following proposition is a standard application of ergodicity.

\begin{prop}\label{Prop:EE}
Let $\mu$ be an $\Out(F_n)$-invariant, ergodic measure on $\G_n$. For any countable set $F$ of sentences in $\L_{\omega_1, \omega}$, there exists a Borel subset $A_F\subseteq \G_n$ of measure $\mu(A_F)=1$ such that, for any $(G,X), (H,Y)\in A_F$, we have $G\equiv_F H$. In particular, there exists a Borel subset $A\subseteq \G_n$ such that $\mu(A)=1$ and, for any $(G,X), (H,Y)\in A$, the groups $G$ and $H$ are elementarily equivalent.
\end{prop}

\begin{proof}
Fix some $n\in \NN$ and an ergodic, $\Out(F_n)$-invariant probability measure $\mu$ on $\G_n$.
For every sentence $\sigma \in \L_{\omega_1, \omega}$, we define its set of models
$$
\Mod_n(\sigma)=\{ (G,X)\in \G_n \mid G\models \sigma\}.
$$
It is well-known that $\Mod_n(\sigma)$ is a Borel subset of $\G_n$ for every $\sigma \in \L_{\omega_1, \omega}$. For instance, this follows immediately from \cite[Proposition 5.1]{Osi21}.

For any $\alpha \in \Out(F_n)$ and any $(G,X)\in \G_n$, we have $\alpha(G,X)=(G, X^\prime)$ for some $X^\prime$. Therefore, $\Mod_n(\sigma)$ is $\Out(F_n)$-invariant. By ergodicity, we have
\begin{equation}\label{Eq:01}
\mu (\Mod_n(\sigma))\in \{ 0, 1\}.
\end{equation}

Let $T$ denote the set of all sentences $\sigma \in F$ such that $\mu (\Mod_n(\sigma)) =1$ and let
$$
A_F=\{ (G,X)\in \G_n \mid  G\models \sigma \textrm{ for all } \sigma \in T\} = \bigcap_{\sigma\in T}\Mod_n(\sigma).
$$
Since $F$ is countable, so is $T$ and, therefore, $\mu (A_F)=1$. We claim that, for any $(G,X)\in A_F$ and any $\sigma \in F$, we have $G\models \sigma$ if and only if $\sigma \in T$.

Indeed, if $\sigma \in T$, then $G\models \sigma$ by the definition of $A_F$. If $\sigma \notin T$, then $\mu(\Mod_n(\sigma))=0$ by (\ref{Eq:01}) and $\mu(\Mod_n(\lnot\sigma))=\mu (\G_n \setminus \Mod_n(\sigma))= 1$. Therefore, $\lnot \sigma \in T$ and  $G\models \lnot\sigma$; equivalently, $G$ does not satisfy $\sigma$. Thus, all groups whose markings belong to $A_F$ are $F$-equivalent.
\end{proof}

\begin{rem}\label{Rem:c}
If $\mu$ is non-atomic, every Borel subset $A\subseteq \G_n$ of measure $1$ must have the cardinality of the continuum. (Recall that the continuum hypothesis holds for Borel sets as proved by Alexandrov and Hausdorff, see \cite[Theorem 3.16]{Kech}.) Since the isomorphism class $[G]_n=\{(H,Y)\in \G_n\mid H\cong G\}$ of every finitely generated group $G$ in $\G_n$ is at most countable, such a subset $A$ must contain $2^{\aleph_0}$ pairwise non-isomorphic, elementarily equivalent groups.
\end{rem}

\paragraph{4.2. An application.} Recall that the \emph{support} of a measure $\mu$ on a topological space $X$, denoted by $supp (\mu)$, is the set of all points $x\in X$ such that every open neighborhood of $x$ has positive measure. For example, the support of the product measure on  $2^\NN$ corresponding to any non-trivial probability distribution on $\{ 0,1\}$ is the whole space $2^\NN$.

Proposition \ref{Prop:EE} imposes strong restrictions on the possible supports of $\Out(F_n)$-invariant, ergodic probability measures on $\G_n$. We illustrate this by deriving the following.

\begin{cor}\label{Cor:H}
Let $\mathcal H_n$ denote the subset of $\G_n$ consisting of all pairs $(G,X)$ such that $G$ is non-elementary hyperbolic. For any $n\ge 2$, there is no $\Out(F_n)$-invariant, ergodic probability measure $\mu$ on $\G_n$ whose support contains $\mathcal H_n$.
\end{cor}

\begin{rem}
It is well-known and easy to prove that $\mathcal H_n$ has no isolated points (see, for example, \cite{Osi21}). Our interest in Corollary~\ref{Cor:H} stems from the fact that all measures constructed in the previous section are supported on $\overline{\mathcal H}_n$ for appropriate $n\in \NN$. The proof of this fact is rather technical and we do not provide it here. The main idea generalizes the well-known observation that the ordinary wreath product $\ZZ {\rm \,wr\,}\ZZ$ can be approximated by groups $\ZZ/k\ZZ {\rm \,wr\,}\ZZ$ as $k\to \infty$, and the latter groups are limits of hyperbolic groups obtained by ``truncating" the standard presentation (see, for example, Lemmas 3.1 and 3.2 in \cite{Osi01}). Furthermore, the supports of our measures are ``very small" (in particular, proper) subsets of $\overline{\mathcal H}_n$. Corollary~\ref{Cor:H} shows that, in a certain sense, this is unavoidable.
\end{rem}

\begin{proof}[Proof of Corollary \ref{Cor:H}]
We will need two auxiliary hyperbolic groups $H_1$ and $H_2$ defined as follows.
Let
$$Q=\langle a, b\mid aba^2b\ldots a^{100}b=1\rangle.$$
It is straightforward to verify that the latter presentation satisfies the $C^\prime(1/6)$ condition and, therefore, $Q$ is hyperbolic. Furthermore, $Q$ is not virtually cyclic as it surjects on $\mathbb Z/2\mathbb Z\oplus \mathbb Z/2\mathbb Z$.
Let $H_1$ is the following extension of $\ZZ/3\ZZ=\langle c\mid c^3=1\rangle$ by $Q$:
$$
H_1=\langle a, b, c \mid c=aba^2b\ldots a^{100}b,\; c^3=1,\; a^{-1}ca=b^{-1}cb=c^{-1}\rangle .
$$
Since hyperbolicity is invariant under taking extensions with finite kernel, $H_1$ is also hyperbolic. Note that $H_1$ is generated by two elements, namely $a$ and $b$.

Further, let $H_2$ be a $2$-generated non-elementary hyperbolic group with trivial abelianization. Such a group is easy to define explicitly by a presentation satisfying the $C^\prime(1/6)$ condition; alternatively, such a group exists by \cite[Corollary 3.24]{CIOS23}.

Fix some $n\ge 2$ and some generating sets $X_1$ and $X_2$ of $H_1$ and $H_2$, respectively, so that $(H_i, X_i)\in \G_n$ for $i=1,2$. Since $H_1$ is finitely presented, there exists an open neighborhood $U$ of $(H_1, X_1)$ such that, for any $(G,X)\in U$, $G$ is a quotient of $H_1$ (see, for example, \cite{Gri84} or \cite[Lemma 2.3]{CG}) and the images of $c$ and $c^{-1}$ in $G$ are distinct. The action of the image of $a$ on the image of $\langle c\rangle$ in $G$ yields a non-trivial homomorphism $G\to \ZZ/2\ZZ$.

Similarly, there is an open neighborhood $V$ of $(H_2, X_2)$ such that, for every $(H,Y)\in V$, $H$ is a quotient of $H_2$. Obviously, every such $H$ has trivial abelianization since so does $H_2$.

Although having trivial abelianization is not a first order property, it can be axiomatized in $\L_{\omega_1, \omega}$. Specifically,  let $\sigma$ denote the sentence
$$
\forall\, g\;\,  \bigvee\limits_{k=1}^\infty \Big(\exists\, a_1\, \ldots \, \exists\, a_{2k}\;\, g=[a_1, a_2]\cdots [a_{2k-1}, a_{2k}]\Big),
$$
where $[a,b]$ is the abbreviation of $a^{-1}b^{-1}ab$. Since the commutant of every group is generated by commutators, a group $G$ satisfies $\sigma$ is and only if $G=[G,G]$.

Let $\mu$ be an $\Out(F_n)$-invariant, ergodic probability measure on the space $\G_n$. By Proposition \ref{Prop:EE} applied to $F=\{ \sigma\}$, there exists a subset $A_F\subseteq \G_n$ of measure $\mu(A_F)=1$ such that all marked groups from $A_F$ simultaneously satisfy $\sigma$ or $\lnot \sigma$. If both $(H_1, X_1)$ and $(H_2, X_2)$ belong to $supp(\mu)$, we have $\mu(U)>0$ and $\mu(V)>0$. Hence, $U\cap A_F\ne \emptyset$ and $V\cap A_F\ne \emptyset$, which yields a contradiction. Thus, $(H_1, X_1)$ and $(H_2, X_2)$ cannot simultaneously belong to $supp(\mu)$.
\end{proof}

Our proof of Corollary \ref{Cor:H} essentially relies on the existence of torsion in $H_1$ and $H_2$. We do not know the answer to the following.

\begin{q}\label{q}
Let $n\ge 2$. Does there exist a non-atomic, $\Out(F_n)$-invariant, ergodic probability measure $\mu$ on $\G_n$ such that $supp(\mu)$ contains all non-cyclic, torsion free, $n$-generated, hyperbolic marked groups?
\end{q}

It would also be interesting to know whether there exist non-atomic, $\Out(F_n)$-invariant, ergodic probability measures on $\G_n$ concentrated on other natural classes of groups such as solvable groups, groups of intermediate growth, etc. Grigorchuck and Kropholler indicated to the author that the answer to this question might be positive for solvable groups of derived length $3$.

\vspace{1cm}

\noindent \textbf{Denis Osin: } Department of Mathematics, Vanderbilt University, Nashville 37240, U.S.A.\\
E-mail: \emph{denis.v.osin@vanderbilt.edu}

\end{document}